\documentclass[10pt]{article}
\usepackage[T1]{fontenc}
\usepackage[utf8]{inputenc}
\usepackage{amsmath,times,mathptmx}
\usepackage{amsmath,amssymb,amsthm,mathtools}
\usepackage{bm}
\usepackage{geometry}
\usepackage{microtype}
\usepackage{enumitem}
\usepackage{booktabs,array,tabularx}
\usepackage{xcolor}
\usepackage{tikz}
\usetikzlibrary{shapes.geometric}
\usepackage{hyperref}
\usepackage{fancyhdr}
\usepackage{cleveref}
\usepackage{indentfirst}
\usepackage{amsmath,amsthm} 
\hypersetup{colorlinks=true,linkcolor=cyan!55!black,citecolor=cyan!55!black,urlcolor=cyan!55!black}
\allowdisplaybreaks
\numberwithin{equation}{section}

\setlist[enumerate]{leftmargin=2.1em,label=(\roman*),itemsep=0.15em,topsep=0.3em}
\newtheorem{theorem}{Theorem}[section]
\newtheorem{lemma}{Lemma}
\newtheorem{claim}{Claim}
\theoremstyle{definition}

\newtheorem{problem}{Problem}

\newcommand{\Gcal}{\mathcal{G}}

\begin{document}
\thispagestyle{empty}
\enlargethispage{2\baselineskip}
\title{Independence number, essential connectivity and the distance spectral radius of graphs
\footnote{Supported by Tianshan Talent Training Program (No. 2024TSYCCX0013), Natural Science Foundation of Xinjiang Uygur Autonomous Region (No. 
2024D01C41), the Basic scientific research in universities of Xinjiang Uygur Autonomous Region (XJEDU2025P001) and NSFC (No. 12361071).}}
\author{{Shuang Ding, Dan Li\thanks{Corresponding author. E-mail: ldxjedu@163.com.},  Yuanyuan Chen}\\
{\footnotesize College of Mathematics and System Science, Xinjiang University, Urumqi 830046, China}}
\date{}

\maketitle {\flushleft\large\bf Abstract:}
An independent set of a graph $G$  is a subset of $V_G$, no two of which are adjacent. The cardinality of a maximum independent set in a graph $G$ is called the
independence number of $G$, denoted by $\alpha(G)$. The essential connectivity $\kappa'(G)$ of a graph $G$ is denoted as the minimum number of vertices of $G$ whose removal produces a disconnected graph with at least two non-trivial components. In this paper , we determine the $n$-vertex connected graphs with given independence number and essential connectivity that attain the minimum distance spectral radius, and fully characterize the corresponding extremal graphs.

\vspace{0.1cm}
\begin{flushleft}
\textbf{Keywords:}  Distance spectral radius; Independence number; Essential connectivity

\end{flushleft}
\textbf{AMS Classification:} 05C50; 05C35

\section{Introduction}\label{sec:1}
In this paper, we only consider simple, connected and undirected graphs. Let $G=(V(G),E(G))$ be a  graph with vertex set 
$V(G)$(or simply $V_G$)=$\{u_1,u_2,..., u_n\}$ and edge set $E(G)$(or simply $E_G$). The length of a shortest path connecting vertex \(u_i\) and vertex \(u_j\) is called the distance between \(u_i\) and \(u_j\), denoted $d_G(u_i,u_j)$ (or simply \(d_{ij}\)). The distance matrix of $G$, denoted by $D(G) = (d_{ij})$, is an $n\times n$  real symmetric matrix with zeros on the diagonal,  whose $(i,j)$-entry is $d_G(u_i,u_j)$ (or \(d_{ij}\)). Hence,  all its eigenvalues are real. The largest eigenvalue $\lambda_1(G)$ of $D(G)$ is called the distance spectral radius of $G$. In addition, since $G$ is connected,  by the Perron-Frobenius theorem, there exists a unique positive unit eigenvector $X$ of $D(G)$ corresponding to $\lambda_1(G)$ such that $D(G)X = \lambda_1(G)X$, which is called the Perron vector of $D(G)$.

A subset $F \subseteq V_G$ is called an independent set of $G$ if no two vertices of $F$ are adjacent in $G$. The cardinality of a maximum independent set in a graph $G$ is called the
independence number of $G$, denoted by $\alpha(G)$. A number of results concerning the independence number of graphs. Feng, Yu and Ili\'c \cite{L.H. Feng}  studied the Laplacian spectral radius of unicyclic graphs with given independence number and characterized the extremal graphs completely. Choi and Park \cite{J. Choi} determined the connected graphs of order $n$ achieving the minimal adjacency spectral radius when the independence number equals $\left\lceil\frac{n}{2}\right\rceil-1$. Liu and Wang \cite{X.C. Liu} characterized the graphs attaining the maximum and minimum $A_\alpha$ spectral radius among $n$-vertex connected graphs with independence number $n-4$ for $\alpha\in\left[\frac12,1\right)$. Additional results are available in \cite{C.M. Conde,Y.Q. Cui,Z.Z. Lou}.
The (vertex) connectivity $\kappa(G)$ of a graph $G$ is the minimum number of vertices whose removal disconnects the graph or leaves a single vertex; the edge connectivity $\lambda(G)$ is the minimum number of edges whose removal disconnects the graph. Although these parameter are pivotal metrics for evaluating network reliability, they often underestimates the inherent robustness of large networks. To overcome this limitation, Harary \cite{F. Harary} put forward the notion of conditional connectivity by requiring certain restrictions on the components of $G-T$, where $T$ denotes a vertex subset or an edge subset of $G$. In particular, the essential connectivity $\kappa'(G)$ of $G$ is denoted as the minimum number of vertices of $G$ whose removal produces a disconnected graph with at least two non-trivial components. Many scholars are interested in the study of the extremal problem  regarding the conditional connectivity of graphs. Wang and Li \cite{Y. Wang} gave a sharp upper bound for the cyclic edge connectivity of graphs with given minimum degree, and characterized the extremal graphs attaining the maximum adjacency spectral radius for graphs with prescribed order, minimum degree and cyclic edge connectivity. Fan, Gu and Lin \cite{D.D. Fan} characterized the graphs and strongly connected digraphs achieving maximum adjacency spectral radius for given $l$-connectivity, minimum degree and $l$-edge-connectivity.  For more details, see \cite{ W.X. Ding, W}.

The study of the distance spectral radius of graphs is a crucial topic in the theory of graph spectra. In particular, the problem of determining graphs with prescribed parameters achieving the maximum or minimum distance spectral radius has attracted many researchers' attention. Lin and Feng \cite{H.Q. Lin} investigated the distance spectral radius of connected graphs with prescribed independence number, derived a lower bound, and completely characterized the graphs achieving maximum distance spectral radius for several ranges of independence number.
Hu, Lin, and Zhang \cite{Y.L. Hu} characterized the extremal graphs that minimize the distance spectral radius over connected graphs and strongly connected digraphs with a prescribed $h$-extra $l$-component connectivity. They further studied the case of graphs with given $l$-connectivity and minimum degree.
Zhang, Li and Ding \cite{D.X. Zhang} characterized the extremal graphs achieving the minimum distance spectral radius for connected graphs with prescribed essential connectivity and minimum degree, as well as for strongly connected digraphs with given essential connectivity. Zhang, Li and Gutman \cite{M.J. Zhang} characterized the graphs attaining the minimum distance spectral radius among $n$ vertex graphs with given connectivity and diameter, as well as those with given connectivity and independence number, and recovered several known results as corollaries. For more results, 
see \cite{D. Fan,Z.G. Zhang,Y,H.Y. Lin,J.H. Xin}. Motivated by the distance spectral results discussed above, we turn our attention to the extremal problem involving essential connectivity and independence number. This naturally leads to the following problem.
\begin{problem}
Which graphs attain the minimum distance spectral radius among all connected graphs of order $n\ge k' + \alpha + 2$ with fixed essential connectivity $k'$ and independence number $\alpha\ge2$?
\end{problem}
Let $\Gcal^\alpha_{n,k'}$ denote the family of all $n$ vertex graphs with essential connectivity $\kappa'(G)=k'$ and independence number $\alpha(G)=\alpha$. As usual,  $K_n$ denote the complete graph on $n$ vertices. For any two disjoint graphs $G$ and $H$, the union of graphs $G$ and $H$ is defined by
$G\cup H=\bigl(V_G\cup V_H,\,E_G\cup E_H\bigr)$.
The join of graphs $G$ and $H$, denoted by $G\vee H$, is the graph with vertex set $V(G\vee H)=V_G\cup V_H$,
and edge set $E(G\vee H)=E_G\cup E_H\cup\{wu\mid w\in V_G,\,u\in V_H\}$. In this paper, we provide the solution to Problem 1. Our main result is as follows.
\begin{theorem}\label{thm:1}
Let $G \in \Gcal^{\alpha}_{n,k'}$ with  $\alpha\ge 2$.
\begin{enumerate}[label=(\roman*),ref=(\roman*)]
\item\label{thm:main:i} If $n \ge k' + \alpha + 3$, then
$\lambda_1(G) \ge \lambda_1(K_{k'} \vee ((K_{n-k'-\alpha-1} \vee (\alpha-1)K_1) \cup K_2))$,
the equality holds if and only if $G\cong K_{k'} \vee((K_{n-k'-\alpha-1} \vee (\alpha-1)K_1)\cup K_2)$.

\item\label{thm:main:ii} If $n = k'+\alpha+2$, then
$\lambda_1(G)\geq\lambda_1(K_{k'} \vee ((K_1 \vee \lfloor \alpha/2 \rfloor K_1) \cup (K_1 \vee \lceil \alpha/2 \rceil K_1)))$,
the equality holds if and only if $G\cong K_{k'} \vee ((K_1 \vee \lfloor \alpha/2 \rfloor K_1) \cup (K_1 \vee \lceil \alpha/2 \rceil K_1))$.
\end{enumerate}
\end{theorem}
This paper is organized as follows. Some preliminary concepts and lemmas are presented in Section \ref{sec:2}. Sections \ref{sec:3} and \ref{sec:4} are dedicated to proving parts \ref{thm:main:i} and \ref{thm:main:ii} of Theorem \ref{thm:1}, respectively.

\section{Preliminaries} \label{sec:2}

The {\it Wiener index}\/ of a connected graph $G$ of order $n$ is defined as
\[
W(G)=\sum_{i<j} d_{ij}.
\]
Then we have
\begin{align*}
\lambda_1(D(G))=\max_{\mathbf{x}\in\mathbb{R}^n}\frac{\mathbf{x}^T D(G)\mathbf{x}}{\mathbf{x}^T\mathbf{x}}
\ge \frac{\mathbf{1}^T D\mathbf{1}}{\mathbf{1}^T\mathbf{1}}=\frac{2W(G)}{n},
\tag{1}\label{eq:1}
\end{align*}
where $\mathbf{1}=(1,1,\dots,1)^T$.

In what follows, the next three lemmas will be needed.
\begin{lemma}[{\cite{C.D. Godsil}}]\label{lemma:1}
Let $e$ be an edge of $G$ such that $G-e$ is connected. Then $\lambda_1(G)<\lambda_1(G-e)$.
\end{lemma}
Let $B$ be an $n\times n$ real matrix and $V=\{1,2,\dots,n\}$. Given a partition $\Pi=\{V_1,V_2,\dots,V_s\}$, where $V=V_1\cup V_2\cup\cdots\cup V_s$, the matrix $B$ can be written in the block form
\[
B=
\begin{pmatrix}
B_{11} & B_{12} & \cdots & B_{1s}\\
B_{21} & B_{22} & \cdots & B_{2s}\\
\vdots & \vdots & \ddots & \vdots\\
B_{s1} & B_{s2} & \cdots & B_{ss}
\end{pmatrix}.
\]
The quotient matrix of $B$ with respect to the partition $\Pi$ is the $s\times s$ matrix $M_{\Pi}(B)=(b_{ij})_{i,j=1}^{s}$, where $b_{ij}$ denotes the average row‑sum of $B_{ij}$. If every block $B_{ij}$ has constant row‑sum $b_{ij}$, then the partition $\Pi$ is called an equitable partition. In particular, if $\Pi$ is an equitable partition of $B$, then $M_{\Pi}(B)$ is called the equitable quotient matrix of $B$.
\begin{lemma}[{\cite{L.H. You}}]\label{lemma:2}
Let $B$ be a real symmetric matrix and $\lambda_1(B)$ denote the largest eigenvalue of $B$. If $M_{\Pi}(B)$ is an equitable quotient matrix of $B$, then every eigenvalue of $M_{\Pi}(B)$ is also an eigenvalue of $B$. In particular, if $B$ is a nonnegative irreducible matrix, then $\lambda_1(B)=\lambda_1(M_{\Pi}(B))$.
\end{lemma}

\begin{lemma}[{\cite{M.N. Ellingham}}]\label{lemma:3}
Let $A$ be a real symmetric $n\times n$ matrix and $\lambda$ an eigenvalue of $A$ with an eigenvector $X$ all of whose entries are non-negative. If $S_i(A)$ $(1\leq i\leq n)$ is the $i$th row sum of $A$, then
\[
\min_{1\leq i\leq n}S_i(A)\leq\lambda\leq\max_{1\leq i\leq n}S_i(A).
\]
Moreover, if the row sums of $A$ are not all equal and if all entries of $X$ are positive, then both inequalities above are strict.
\end{lemma}

\section{Proof of Theorem \ref{thm:1}~\ref{thm:main:i}: the case  $n\ge k'+\alpha+3$. } \label{sec:3}

In this section, we determine the unique graphs in the set $\Gcal^\alpha_{n,k'}$ with integer $n\ge k'+\alpha+3$, having minimum distance spectral radius.

\begin{proof}[\bfseries Proof of Theorem \ref{thm:1}~\ref{thm:main:i}]
 Suppose that $G^*\in\Gcal^{\alpha}_{n,k'}$  is a graph that attains the minimum distance spectral radius, where $n\ge k'+\alpha+3$.  Let $S'$ be an independent set of $G^{*}$ with $|S'|=\alpha$. Since $n-k' \ge \alpha+3\ge 5$, by the definition of essential connectivity, there exists some nonempty subset $T'\subseteq V_{G^*}$ with $|T'|=k'$ such that $G^{*} -T'$
 contains at least two non-trivial connected components and let $C_1,C_2,...,C_l$ be the $l$ connected components of $G^{*} -T'$, where $l \ge 2$. For convenience, we set $T'\cap S'=\{w_1',w_2',\ldots,w_g'\}$, $ T'\setminus S'=\{w_1,w_2,\ldots,w_h\}$, $|V_{C_i}\cap S'|=a_i$, $|V_{C_i}\setminus S'|=b_i$ for $1\le i\le l$.  In the following, we provide four claims to determine the structure of the extremal graph $G^*$ to complete the proof of Theorem  \ref{thm:1}~\ref{thm:main:i}. 

\begin{claim}\label{claim:1}
 $C_i\cong K_{b_i}\vee a_iK_1$, $G^*[T']\cong K_h\vee gK_1$ and $G^*[V_{C_i}\cup T']\cong K_{h+b_i}\vee(g+a_i)K_1$ for $1\leq i\leq l$.
\end{claim}

We only verify $C_i\cong K_{b_i}\vee a_iK_1$ and the rest can be proved similarly. If there exist vertices $w\in V_{C_i} \setminus S'$  and $v\in V_{C_i} \cap S'$ such that $vw\notin E_G^*$, then $S'$  is still an independent set of $G^*+{vw}$. Obviously, $G^*+{vw}\in\Gcal^{\alpha}_{n,k'}$ and by Lemma \ref{lemma:1}, $\lambda_1(G^*+{vw})<\lambda_1(G^*)$, which contradicts the minimality of $\lambda_1(G^*)$. With similar reasoning, we have $wu\in E_G^*$ for any $w,u\in V_{C_i}\setminus S'$. Thus, $C_i\cong K_{b_i}\vee a_iK_1$. 
\begin{claim}\label{Claim:2}
 $l=2.$
\end{claim}
 If not, $l\ge3$. Without loss of generality, let $C_1$, $C_2$, and $C_3$ be three connected components of $G^*$, two of which, $C_1$ and $C_2$, are nontrivial. 
 Choose $w_1\in V_{C_1} \setminus S'$ and $w_2\in V_{C_3} $. It is easy to check that $G^*+w_1w_2\in\Gcal^{\alpha}_{n,k'}$. By Lemma \ref{lemma:1}, $\lambda_1(G^*+{w_1w_2})<\lambda_1(G^*)$, a contradiction. The claim holds.
 
  By Claim \ref{Claim:2}, we have $G^* -T' = C_1 \cup C_2$. For convenience, write $V_{C_1}\cap S'=\{u_1',u_2',\ldots,u_{a_1}'\}$, $ V_{C_1}\setminus S'=\{u_1,u_2,\ldots,u_{b_1}\}$, $V_{C_2}\cap S'=\{v_1',v_2',\ldots,v_{a_2}'\}$ and $ V_{G_2}\setminus S'=\{v_1,v_2,\ldots,v_{b_2}\}$.
\begin{claim}\label{claim:3}
Either $C_1$ or $C_2$ is isomorphic to $K_2$.
\end{claim} 
Otherwise, $|V_{C_i}|=a_i+b_i\ge 3$ for $1\leq i\leq 2$. Clearly, $b_i\ge 1$ for $1\leq i\leq 2$. 
Based on whether either of the two connected components contains a vertex from the independent set, we split the proof into three cases.

\noindent\textbf{Case 1.} $a_1 \ge 1$ and $a_2 \ge 1$.

Let $G_1=G^*-\{u_1v\mid v\in V_{C_1}\setminus\{u_1,u_1'\}\}-\{u_1'v\mid v\in V_{C_1}\setminus\{u_1,u_1'\}\}
+\{u_iv\mid 2 \le i \le b_1,\ v\in V_{C_2}\}+\{u_i'v\mid 2 \le i \le a_1,\ v\in V_{C_2}\setminus\ S'\}$ and $G_1'=G^*-\{v_1v\mid v\in V_{C_2}\setminus\{v_1,v_1'\}\}-\{v_1'v\mid v\in V_{C_2}\setminus\{v_1,v_1'\}\}+\{v_iv\mid 2 \le i \le b_2,\ v\in V_{C_1}\}+\{v_i'v\mid 2 \le i \le a_2,\ v\in V_{C_1}\setminus\ S'\}$. Clearly, $G_1\cong G_1' \cong G_2$, where 
$G_2-T'=C_1'\cup C_2'$, $C_1'\cong K_{b_1+b_2-1}\vee (a_1+a_2-1)K_1$ and $C_2'\cong K_2$. Let $V(C_2')=\{u,w\}$, where $w\in S'$. Let $X$ be the Perron vector of $D(G_2)$, with \(x_v\) denoting the entry of $X$ corresponding to vertex $v\in V(G_2)$. By symmetry, set $x_v=x_0$ for $v \in T'\setminus\ S'$, $x_v=x_0'$ for $v \in T'\cap S'$, $x_v=x_1$ for $v \in V_{C_1'}\setminus\ S'$ , $x_v=x_1'$ for $v \in V_{C_1'}\cap S'$, $x_u=x_2$ and $x_w=x_2'$. Next, we consider two subcases according to the value of $h$.

\noindent\textbf{Subcase 1.1.} $h \ge 1$.

From $D(G_2)X=\lambda_1(G_2)X$, we obtain
\begin{align*}
\lambda_1(G_2) x_1' &= 2gx_0'+ hx_0+(b_1+b_2-1)x_1+2(a_1+a_2-2)x_1'+2x_2+2x_2', \\
\lambda_1(G_2) x_1 &= gx_0'+ hx_0+(b_1+b_2-2)x_1+(a_1+a_2-1)x_1'+2x_2+2x_2', \tag{2}\label{eq:2}\\
\lambda_1(G_2)x_2'  &= 2gx_0'+ hx_0+2(b_1+b_2-1)x_1+2(a_1+a_2-1)x_1'+x_2, \tag{3}\label{eq:3}\\
\lambda_1(G_2)x_2  &= gx_0'+ hx_0+2(b_1+b_2-1)x_1+2(a_1+a_2-1)x_1'+x_2', \tag{4}\label{eq:4}\\
\end{align*}
which yield
\begin{align*}
2x_1-x_2'&=\frac{hx_0+3x_2'+3x_2}{\lambda_1(G_2)+1}>0,\tag{5}\label{eq:5}\\
2x_1-x_2&=\frac{gx_0'+hx_0+3x_2'+3x_2}{\lambda_1(G_2)+1}>0,\tag{6}\label{eq:6}
\end{align*}
and
\begin{align*}
&(\lambda_1(G_2)+4)((a_1+a_2)x_1'+(b_1+b_2)x_1-2(x_2+x_2'))\\
&= g(2a_1+2a_2+b_1+b_2-6)x_0'+ h(a_1+a_2+b_1+b_2-4)x_0\\
&+(2a_1+2a_2+b_1+b_2-6)(a_1+a_2-1)x_1'+2x_1'+3x_1\\
& +(b_1+b_2-1)(a_1+a_2+b_1+b_2-5)x_1+(2a_1+2a_2+2b_1+2b_2-10)(x_2'+x_2).
\end{align*}
Recall that $b_i \ge 1$, $a_i \ge 1$, $b_1+b_2=n-\alpha-k'+g \ge 3$ and $a_i+b_i\ge 3$ for $1\leq i\leq 2$. Combining these with $g\ge 0$, we can deduce that $(a_1+a_2)x_1'+(b_1+b_2)x_1-2(x_2+x_2') > 0$. It follows that
\begin{align*}
\frac{1}{2}(\lambda_1(G^*) -\lambda_1(G_2))& \ge \frac{1}{2} X^T (D(G^*) - D(G_2)) X \\ 
&=\frac{1}{2} X^T (D(G^*) - D(G_1)) X\\
&=(a_1-1)(b_2x_1-x_2)x_1'+(b_1-1)(a_2x_1'+b_2x_1-x_2-x_2')x_1\\
&=\frac{1}{2} X^T (D(G^*) - D(G_1')) X\\
&=(a_2-1)(b_1x_1-x_2)x_1'+(b_2-1)(a_1x_1'+b_1x_1-x_2-x_2')x_1.
\end{align*}
Assume that $b_i\ge 2$ for $i=1,2$. Then $b_ix_1-x_2>0$ for $i=1,2$. If $(a_1-1)(b_2x_1-x_2)x_1'+(b_1-1)(a_2x_1'+b_2x_1-x_2-x_2')x_1=(a_2-1)(b_1x_1-x_2)x_1'+(b_2-1)(a_1x_1'+b_1x_1-x_2-x_2')x_1\le 0$ , then $(a_1+a_2)x_1'+(b_1+b_2)x_1-2(x_2+x_2') \le 0$, a contradiction. In the case $b_1=1$, we have $a_1\ge 2$ and $b_2 \ge 2$, which implies $(a_1-1)(b_2x_1-x_2)x_1'+(b_1-1)(a_2x_1'+b_2x_1-x_2-x_2')x_1=(a_1-1)(b_2x_1-x_2)x_1'>0$, a contradiction. A similar contradiction can be obtained when $b_2=1$. Therefore, $\frac{1}{2}(\lambda_1(G^*) -\lambda_1(G_2))>0$ and $\lambda_1(G_2) <\lambda_1(G^*)$, a  contradiction.

\noindent\textbf{Subcase 1.2.} $h = 0$.

Then $g=k'-h \ge 1$. By $D(G_2)X=\lambda_1(G_2)X$, we have
\begin{align*}
\lambda_1(G_2) x_1' &= 2gx_0'+(b_1+b_2-1)x_1+2(a_1+a_2-2)x_1'+3x_2+4x_2',\\
\lambda_1(G_2) x_1 &= gx_0'+(b_1+b_2-2)x_1+(a_1+a_2-1)x_1'+2x_2+3x_2',\\
\lambda_1(G_2)x_2'  &= 2gx_0'+3(b_1+b_2-1)x_1+4(a_1+a_2-1)x_1'+x_2,\\
\lambda_1(G_2)x_2  &= gx_0'+2(b_1+b_2-1)x_1+3(a_1+a_2-1)x_1'+x_2'.
\end{align*}
Thus,
\begin{align*}
x_1'-x_1=\frac{gx_0'+(a_1+a_2-2)x_1'+x_2+x_2'}{\lambda_1(G_2)+1}>0
\end{align*}
and
\begin{align*}
&(\lambda_1(G_2)+2)((a_1+a_2-2)x_1'+(b_1+b_2-1)x_1-x_2)\\
&=g(2a_1+2a_2+b_1+b_2-6)x_0'+(a_1+a_2-1)(2a_1+2a_2+b_1+b_2-8)x_1'\\
&+(b_1+b_2-1)(a_1+a_2+b_1+b_2-4)x_1+(4a_1+4a_2+3b_1+3b_2-12)x_2'\\
&+(3a_1+3a_2+2b_1+2b_2-10)x_2>0\quad (\text{since } \ a_i+b_i\ge 3 \  \text{for } \ 1\leq i\leq 2).\\
\end{align*}
It indicates that $(a_1+a_2-2)x_1'+(b_1+b_2-1)x_1-x_2>0$, and hence
\begin{align*}
x_2'-x_1'=\frac{2(a_1+a_2-2)x_1'+2(b_1+b_2-1)x_1-2x_2}{\lambda_1(G_2)+4}>0.
\end{align*}
For convenience, set $a_1+a_2-1=A,$ and $b_1+b_2-1=B$. Then
\begin{align*}
&(\lambda_1(G_2)+6)((a_1+a_2+2)x_1'+(b_1+b_2)x_1-2(x_2+2x_2'))\\
&=(\lambda_1(G_2)+6)((A+3)x_1'+(B+1)x_1-2(2x_2'+x_2))\\
&=g(2A+B-3)x_0'+(2A^2+AB-11A+12)x_1'+(AB+B^2-7B+5)x_1\\
&+(4A+3B-11)x_2'+(3A+2B-5)x_2. \tag{7}\label{eq:7}
\end{align*}
 Note that $A\ge 1$, $B\ge n-\alpha-k'+g \ge 3$ and $A+B \ge 4$. Clearly, $2A+B-3>0$, $4A+3B-11>0$ and $3A+2B-5>0$. In addition, $2A^2+AB-11A+12>0$ for $A=1,2,3$. When $A\ge 4$, we obtain $2A^2+AB-11A+12>2(A-2)(A-3)>0.$
 It is easy to see that $AB+B^2-7B+5 \ge 0$ whenever $A+B\ge 6$. For $4 \le A+B \le 5$, all pairs $(A,B)$ satisfying $AB+B^2-7B+5<0$ are listed as $(1,4),\;(2,3)$ and $(1,3)$.
Combined this with $x_1'>x_1$, we obtain $\eqref{eq:7} >0$ by successive calculations in $\eqref{eq:7}$.
 From the preceding discussion, it follows that 
 $(a_1+a_2+2)x_1'+(b_1+b_2)x_1-2(x_2+2x_2')>0.$ Without loss of generality, assume that $a_1x_1'+ b_1x_1 \ge  a_2x_1'+ b_2x_1$.Thus, $a_1x_1'+ b_1x_1>x_2+2x_2'-x_1'$. Furthermore, $2a_1x_1'+ b_1x_1 \ge x_1'+a_1x_1'+ b_1x_1>x_2+2x_2'$. Since $x_2'>x_1'$, we obtain $(a_1+a_2)x_1'+(b_1+b_2)x_1>2(x_2+x_2')$. Then, $a_1x_1'+ b_1x_1 > x_2+x_2'$. Moreover,
\begin{align*}
\frac{1}{2}(\lambda_1(G^*) -\lambda_1(G_2))& \ge \frac{1}{2} X^T (D(G^*) - D(G_2)) X \\ 
&=2(a_1-1)x_1'(a_2x_1'+b_2x_1-x_2-x_2')\\
&+(b_1-1)x_1(2a_2x_1'+b_2x_1-2x_2'-x_2)\\
&=2(a_2-1)x_1'(a_1x_1'+b_1x_1-x_2-x_2')\\
&+(b_2-1)x_1(2a_1x_1'+b_1x_1-2x_2'-x_2)\\
&>0.
\end{align*}
Thus, $\lambda_1(G_2) <\lambda_1(G^*)$, a  contradiction.

\noindent\textbf{Case 2.} $\min\{a_1,a_2\}=0$ and $a_1+a_2 \ge 1$.

 Without loss of generality, assume that $a_1 > a_2=0$. Then, $b_2 \ge 3$  and $C_2 \cong K_{b_2}$. we assert that $g \ge 1$.  If instead $g=0$, then we deduce $\alpha(G^*)> \alpha$, a contradiction. Subsequently, we also distinguish two subcases according to the value of $h$.

\noindent\textbf{Subcase2.1.} $h \ge 1$.

Note that $a_1+b_1\ge 3$. We begin by considering $a_1+b_1\ge 4$. Let $G_3=G^*-\{v_1v\mid v\in V_{C_2}\setminus\{v_1,v_2\}\}-\{v_2v\mid v\in V_{C_2}\setminus\{v_1,v_2\}\}+\{v_iv\mid 3 \le i \le b_2,\ v\in V_{C_1}\}.$ Let $Y$ denote the Perron vector of $D(G_3)$, with \(y_v\) denoting the entry of $Y$ corresponding to vertex $v\in V(G_3)$. 
By symmetry, set $y_v=y_0$ for $v \in T'\setminus\ S'$, $y_v=y_0'$ for $v \in T'\cap S'$, $y_v=y_1$ for $v \in (V_{C_1}\setminus\ S')\cup (V_{C_2} \setminus \{v_1,v_2\})$, $y_v=y_1'$ for $v \in V_{C_1}\cap S'$, and $y_{v_1}=y_{v_2}=y_2$. Then
\begin{align*}
\lambda_1(G_3) y_1 &= gy_0'+hy_0+(b_1+b_2-3)y_1+a_1y_1'+4y_2,\tag{8}\label{eq:8}\\
\lambda_1(G_3)y_1' &= 2gy_0'+hy_0+(b_1+b_2-2)y_1+2(a_1-1)y_1'+4y_2,\\
\lambda_1(G_3)y_2  &= gy_0'+hy_0+2(b_1+b_2-2)y_1+2a_1y_1'+y_2,\tag{9}\label{eq:9}
\end{align*}
which indicates that 
\[y_1'-y_1=\frac{gy_0'+(a_1-1)y_1'}{\lambda_1(G_3)+1}>0\]
and 
\[2y_1-y_2=\frac{gy_0'+hy_0+6y_2}{\lambda_1(G_3)+1}>0.\]
Combining these with  $a_1+b_1 \ge 4$ and $b_2 \ge 3$, we can deduce that
\begin{align*}
\frac{1}{2}(\lambda_1(G^*) -\lambda_1(G_3))& \ge \frac{1}{2} Y^T (D(G^*) - D(G_3))Y \\ 
&=(b_2-2)(a_1y_1'+b_1y_1-2y_2)y_1\\
&>2(2y_1-y_2)y_1\quad(\text{since} \ y_1'>y_1)\\
&>0,
\end{align*}
and hence, $\lambda_1(G_3)< \lambda_1(G^*)$ , a contradiction. Next, we treat the case $a_1+b_1= 3$. The only admissible ordered pairs are $(a_1,b_1)=(2,1)$ and $(1,2)$.
If $(a_1,b_1)=(2,1)$, then
\begin{align*}
\frac{1}{2}(\lambda_1(G^*) -\lambda_1(G_1))& \ge \frac{1}{2} X^T (D(G^*) - D(G_1)) X \\ 
&=(b_2x_1-x_2)x_1'\\
&>0
\end{align*}
and $\lambda_1(G_1)< \lambda_1(G^*)$, a contradiction. When $(a_1,b_1)=(1,2)$ , by \eqref{eq:2} -- \eqref{eq:4} , it follows that
\[b_2x_1-x_2-x_2'=\frac{hx_0+((b_2-3)\lambda_1(G_1)+4b_2-4)x_1}{\lambda_1(G_1)+5}>0.\]
Combining this with \eqref{eq:5} and \eqref{eq:6}, we can get
\begin{align*}
\frac{1}{2}(\lambda_1(G^*) -\lambda_1(G_1))& \ge \frac{1}{2}X^T (D(G^*) - D(G_1))X \\ 
&=(b_2x_1-x_2-x_2')x_1\\
&>0 .
\end{align*}
Thus, $\lambda_1(G_1)< \lambda_1(G^*)$, a contradiction. 

\noindent\textbf{Subcase 2.2.} $h = 0$.

By $D(G_3)Y=\lambda_1(G_3)Y$, we have
\begin{align*}
\lambda_1(G_3) y_1 &= gy_0'+(b_1+b_2-3)y_1+a_1y_1'+4y_2,\\
\lambda_1(G_3)y_1' &= 2gy_0'+(b_1+b_2-2)y_1+2(a_1-1)y_1'+6y_2,\\
\lambda_1(G_3)y_2  &= gy_0'+2(b_1+b_2-2)y_1+3a_1y_1'+y_2.
\end{align*}
And hence,
\[y_1'-y_1=\frac{gy_0'+(a_1-1)y_1'+2y_2}{\lambda_1(G_3)+1}>0\]
and
\[y_1'+y_1-y_2=\frac{2gy_0'+7y_1'+8y_1}{\lambda_1(G_3)+9}>0.\]
It follows that
\begin{align*}
\frac{1}{2}(\lambda_1(G^*) -\lambda_1(G_3))& \ge \frac{1}{2} Y^T (D(G^*) - D(G_3))Y \\ 
&=(b_2-2)(2a_1y_1'+b_1y_1-2y_2)y_1\\
&>0\quad(\text{since} \ a_1 \ge 1,\ b_1 \ge 1 \text{ and } a_1+b_1\ge 3),
\end{align*}
and $\lambda_1(G_3)< \lambda_1(G^*)$, a contradiction. 

\noindent\textbf{Case 3.} $a_1 = a_2 = 0$.

In this case, $g=\alpha \ge 2$ and $b_i\ge 3$ for $i=1,2$. Without loss of generality, assume $b_1 \ge b_2$.  From \eqref{eq:8} and \eqref{eq:9}, we obtain 
\[(b_1+b_2)y_1-4y_2=\frac{(b_1+b_2-4)(gy_0'+hy_0+2(b_1+b_2-2)y_1)}{\lambda_1(G_3)+b_1+b_2-1}>0.\]
Thus,
\begin{align*}
\frac{1}{2}(\lambda_1(G^*) -\lambda_1(G_3))& \ge \frac{1}{2} Y^T (D(G^*) - D(G_3))Y \\ 
&=(b_2-2)(b_1y_1-2y_2)y_1\\
&>0\quad(\text{since}\ 2b_1y_1 \ge (b_1+b_2)y_1 > 4y_2 ).
\end{align*}
Then $\lambda_1(G_3)< \lambda_1(G^*)$, a contradiction. 

Consequently, there exists some $i\in\{1,2\}$ such that $C_i\cong K_2$, which completes the proof of the Claim \ref{claim:3}.
Without loss of generality, we may assume that $C_2\cong K_2$ and $V(C_2)=\{u',v'\}$. Recall that $T'\cap S'=\{w_1',w_2',\ldots,w_g'\}$,
$V_{C_1}\cap S'=\{u_1',u_2',\ldots,u_{a_1}'\}$ and $ V_{C_1}\setminus S'=\{u_1,u_2,\ldots,u_{b_1}\}$. We now establish the following claim.

\begin{claim}\label{claim:4}
$S'\cap T'=\emptyset.$
\end{claim}
Otherwise, suppose that $S'\cap T'\neq \emptyset$ and $g \ge 1$.  When $S'\cap V(C_2) \neq \emptyset$, we have $b_1=n-k'-\alpha-1+g \ge 2+g$  . Suppose $S'\cap V(C_2)=\{u'\}$ and let $G_4=G^*+\{u_iv'\mid 1 \le i \le g \}+\{u_iu'\mid 1 \le i \le g \}-\{w_i'v'\mid 1 \le i\le g\}.$ It is routine to check that $G_4\in\Gcal^\alpha_{n,k'}$. Let $Z$ denote the Perron vector of $D(G_4)$, with \(z_v\) denoting the entry of $Z$ corresponding to vertex $v\in V(G_4)$. By symmetry, set $z_v=z_0$ for $v \in (T'\setminus S')\cup\{u_1,u_2,\ldots,u_g\}$, $z_v=z_1$ for $v \in \{u_{g+1},u_{g+2},\ldots,u_{b_1}\}$, $z_v=z_1'$ for $v \in V_{C_1}\cap S'\cup\{w_1',w_2',\ldots,w_g'\}$ and $z_{u'}=z_{v'}=z_2.$  Then, by $D(G_4)Z=\lambda_1(G_4)Z$, we get
\begin{align*}
\lambda_1(G_4)z_0 &= 2z_2+(k'-1)z_0+(n-k'-\alpha-1)z_1+(\alpha-1)z_1',\\
\lambda_1(G_4)z_1' &= 4z_2+k'z_0+(n-k'-\alpha-1)z_1+2(\alpha-2)z_1'.
\end{align*}
Hence,
\[2z_0-z_1'=\frac{k'z_0+(n-k'-\alpha-1)z_1+z_1'}{\lambda_1(G_4)+1}>0.\]
If $h\ge 1$, then 
\begin{align*}
\frac{1}{2}(\lambda_1(G^*) -\lambda_1(G_4))& \ge \frac{1}{2} Z^T (D(G^*) - D(G_4))Z \\ 
&=gz_2(2z_0-z_1')\\
&>0
\end{align*}
and $\lambda_1(G_4)< \lambda_1(G^*)$, a contradiction; 
while if $h=0$, then $k'=g$,
\begin{align*}
\frac{1}{2}(\lambda_1(G^*) -\lambda_1(G_4))& \ge \frac{1}{2} Z^T (D(G^*) - D(G_4)) Z \\ 
&=z_2(k'(3z_0-z_1')+(n-k'-\alpha-1)z_1+3(\alpha-k'-1)z_1')\\
&>0\quad(\text{since}\ n\ge k'+\alpha+3, \ 3z_0>2z_0>z_1' \text{ and } \alpha \ge k'+1 ).
\end{align*} 
and $\lambda_1(G_4)< \lambda_1(G^*)$, a contradiction. 
When $S'\cap V(C_2)=\emptyset$, we consider the following four possible cases.

\noindent\textbf{Case 1.} $h \ge 2$ and $a_1 \ge 1$.

Recall that $G^* -T' = C_1 \cup C_2$ and $V(C_2)=\{u',v'\}$ and let $G_5=G^*+\{u_1'u',u_1'v'\}-\{w_1u',w_1v'\}$, where $u_1'\in V_{C_1}\cap S'$ and $w_1 \in T'\setminus S'$. Clearly, $G_5 \in \Gcal^\alpha_{n,k'}.$ Let $X$ and $Y$ be the Perron vectors of $D(G^{*})$ and $D(G_{5})$, respectively, where $x_{v}$ ($v\in V(G^{*})$) and $y_{v}$ ($v\in V(G_{5})$) denote the entries of $X$ and $Y$, respectively. By symmetry, set $x_v=x_0$ for $v \in T'\setminus\ S'$, $x_v=x_0'$ for $v \in T'\cap S'$, $x_v=x_1'$ for $v \in V_{C_1}\cap S'$, $x_v=x_1$ for $v \in V_{C_1}\setminus\ S'$ and $x_{u'}=x_{v'}=x_2.$ And let $y_v=y_0$ for $v \in \bigl(T'\setminus S'\bigr)\setminus\{w_1\}$, $y_v=y_0'$ for $v \in \bigl(T'\cap S'\bigr)\cup \{u_1'\}$, $y_v=y_1'$ for $v \in \bigl(V_{C_1}\cap S'\bigr)\setminus \{u_1'\}$, $y_v=y_1$ for $v \in \bigl(V_{C_1}\setminus\ S'\bigr)\cup\{w_1\}$ and $y_{u'}=y_{v'}=y_2.$ Then,
\begin{align*}
\lambda_1(G^*)x_1' &= 2gx_0'+hx_0+2(a_1-1)x_1'+b_1x_1+4x_2,\\
\lambda_1(G^*)x_0 &= gx_0'+(h-1)x_0+a_1x_1'+b_1x_1+2x_2,\\
\lambda_1(G_5)y_0' &= 2gy_0'+(h-1)y_0+2(a_1-1)y_1'+(b_1+1)y_1+2y_2,\\
\lambda_1(G_5)y_1 &= (g+1)y_0'+(h-1)y_0+(a_1-1)y_1'+b_1y_1+4y_2,
\end{align*}
from which we get 
\[x_1'-x_0=\frac{gx_0'+(a_1-1)x_1'+2x_2}{\lambda_1(G^*)+1}\]
and
\[y_1-y_0'=\frac{2y_2-gy_0'-(a_1-1)y_1'}{\lambda_1(G_5)+1}.\]
 If $\lambda_1(G^*) \le \lambda_1(G_5)$,  it follows that 
 $y_2(x_1'-x_0) > \frac{2x_2y_2}{\lambda_1(G_5)+1}> x_2(y_1-y_0')$ and
 \begin{align*}
 (\lambda_1(G^*)-\lambda_1(G_5))X^T Y &= X^T (D(G^*) - D(G_5)) Y \\
 &=2(y_2(x_1'-x_0)-x_2(y_1-y_0')) \\
 &>0,
 \end{align*} 
 a contradiction. Then $\lambda_1(G_5) < \lambda_1(G^*)$, which is still a contradiction. 
   
\noindent\textbf{Case 2.} $h = 0$ and $a_1 \ge 1$.

If $g = 1$, set $S'\cap T'= \{w_1'\}$ and $G_6'= G^*+\{w_1'u \mid u\in V_{C_1}\cap S'\}$. Thus, $G_6'\in \Gcal^\alpha_{n,k'}$. And by Lemma \ref{lemma:1}, $\lambda_1(G_6')<\lambda_1(G^*)$, a contradiction. If $g \ge 2$, let $G_6=G^*+\{u_1u',u_1v'\}-\{w_1'u',w_1'v'\}$, where $u_1\in V_{C_1}\setminus S'$ and $w_1' \in T'\cap S'$. Clearly, $G_6\in \Gcal^\alpha_{n,k'}$.
 Let $Z$ denote the Perron vector of $D(G_6)$, with \(z_v\) denoting the entry of $Z$ corresponding to vertex $v\in V(G_6)$.  By symmetry, set $z_{u_1}=z_0$ , $z_v=z_1'$ for $v \in \bigl(V_{C_1}\cap S'\bigr)\cup\{w_1'\}$, and $z_{u'}=z_{v'}=z_2$ . Then,
\begin{align*}
\frac{1}{2}(\lambda_1(G^*) -\lambda_1(G_6))& \ge \frac{1}{2} Z^T (D(G^*) - D(G_6)) Z \\ 
&=2z_2(z_0+(a_1-1)z_1')\\
&>0
\end{align*}
and $\lambda_1(G_6) < \lambda_1(G^*)$, a contradiction.

 Let $\widetilde{H}\cong K_{k'} \vee ((K_{n-k'-\alpha-1} \vee (\alpha-1)K_1) \cup K_2)$, where $n-k'-\alpha-1 \ge 2$, $\alpha\ge2$, $V(K_{k'})=\{u_1,u_2,\ldots,u_{k'}\}=V_1$, $V(K_{n-k'-\alpha-1})= V_2$, $V((\alpha-1)K_1)= V_3$ and $V(K_2)= V_4$. Obviously, $\widetilde{H}\in \Gcal^\alpha_{n,k'}$. Set $V_1'=V(K_{k'-1})=V_1\setminus \{u_{k'}\}$. From \eqref{eq:1}, we obtain that $\lambda_1(\widetilde{H}) \ge \frac{\alpha^2+4n-4k'-3\alpha-6}{n}+n-1>n-1$. Partition $\Pi_1: V(\widetilde{H}) = V_1 \cup V_2 \cup V_3 \cup V_4 $  and partition
 $\Pi_2: V(\widetilde{H}) =\{u_{k'}\}\cup V_1' \cup  V_2 \cup V_3 \cup V_4 $ are two distinct equitable partitions of $\widetilde{H}$, and their corresponding equitable quotient matrices are
\[
M_{\Pi_1}(\widetilde{H})=
\begin{pmatrix}
k'-1 & n-k'-\alpha-1 & \alpha-1 & 2 \\
k' & n-k'-\alpha-2 & \alpha-1 & 4 \\
k' & n-k'-\alpha-1 & 2(\alpha-2) & 4 \\
k' & 2(n-k'-\alpha-1) & 2(\alpha-1) & 1 
\end{pmatrix}
\]
and
\[
M_{\Pi_2}(\widetilde{H})=
\begin{pmatrix}
0 & k'-1 & n-k'-\alpha-1 & \alpha-1 & 2 \\
1 & k'-2 & n-k'-\alpha-1 & \alpha-1 & 2 \\
1 & k'-1 & n-k'-\alpha-2&\alpha-1 & 4 \\
1 & k'-1 & n-k'-\alpha-1& 2(\alpha-2) & 4 \\
1 & k'-1 & 2(n-k'-\alpha-1) & 2(\alpha-1) & 1 
\end{pmatrix}.
\] 
Set $\det\bigl(x I_{4}-M_{\Pi_1}(\widetilde{H})\bigr)=f(x)$. A direct calculation gives
$\det\bigl(x I_{5}-M_{\Pi_2}(\widetilde{H})\bigr)=(x+1)f(x)$. 

\noindent\textbf{Case 3.} $h \ge 0$ and $a_1 = 0$.

Notice that $h=k'-\alpha$. When $k'-\alpha=0$, the partition $\Pi_3: V(G^*) = (T'\cap S') \cup C_1 \cup C_2 $ 
is an equitable partition of $G^*$, and the corresponding equitable quotient matrix is given by
\[
M_{\Pi_3}(G^*)=
\begin{pmatrix}
2(\alpha-1) & n-k'-2 & 2 \\
\alpha & n-k'-3 & 4 \\
 \alpha & 2(n-k'-2) & 1 
\end{pmatrix}.
\]
In view of Lemma \ref{lemma:2}, one finds that $\lambda_1(G^*)$ is the largest root $\lambda_1(M_{\Pi_3}(G^*))$ of $\det(x I_{3}-M_{\Pi_3}(G^*))= 0$. 
When $k'-\alpha \ge 1$, the partition $\Pi_4: V(G^*) = (T'\setminus S') \cup (T'\cap S') \cup C_1 \cup C_2 $ 
is an equitable partition of $G^*$, and the corresponding equitable quotient matrix is given by
\[
M_{\Pi_4}(G^*)=
\begin{pmatrix}
k'-\alpha-1 & \alpha & n-k'-2 & 2 \\
k'-\alpha &2(\alpha-1) & n-k'-2 & 2 \\
k'-\alpha & \alpha & n-k'-3 & 4 \\
k'-\alpha & \alpha & 2(n-k'-2) & 1 
\end{pmatrix}.
\]
Similarly, $\lambda_1(G^*)$is the largest root $\lambda_1(M_{\Pi_4}(G^*))$ of $\det(x I_{4}-M_{\Pi_4}(G^*))= 0$. 
In addition, when $k'=\alpha$, one can easily get $\det\bigl(x I_{4}-M_{\Pi_4}(G^*)\bigr)=(x+1)(\det\bigl(x I_{3}-M_{\Pi_3}(G^*)\bigr))$ and $\lambda_1(M_{\Pi_3}(G^*))= \lambda_1(M_{\Pi_4}(G^*))$.
Hence, $\lambda_1(G^*)=\lambda_1(M_{\Pi_4}(G^*))$ for all $k'\ge \alpha$. Set $\det\bigl(x I_{4}-M_{\Pi_4}(G^*)\bigr)=h_1(x)$ and $G_1(x)=f(x)-h_1(x).$ A direct calculation gives
\begin{align*}
G_1(x)=f(x)-h_1(x)=x^{3}+A_1x^{2}+A_2x+A_3,
\end{align*}
where $A_1=2\alpha-n+1$, $A_2=22\alpha-6\alpha^{2}+6k'-8n-1$ and $A_3=16\alpha+6k'-7n-8\alpha k'+2\alpha n+2k'n+4\alpha^{2}k'-2\alpha^{2}n-2\alpha^{2}-2k'^{2}-1.$ Thus, $G_1'(x)=3x^2+2A_1x$, $G_1''(x)=6x+2A_1+A_2$ and $G_1'''(x)=6>0$. Furthermore, combining with $\alpha \ge 2$, $k'-\alpha \ge 0$ and $n \ge k'+\alpha+3$, direct computation in  MATLAB R2018b \cite{Inc} yields $G_1'(n-1)>0$, $G_1''(n-1)>0$ and $G_1( \frac{\alpha^2+4n-4k'-3\alpha-6}{n}+n-1)>0$. It follows that $G_1(x)$ is monotone increasing when $x > n-1$. Since $\lambda_1(\widetilde{H}) \ge \frac{\alpha^2+4n-4k'-3\alpha-6}{n}+n-1>n-1$, we conclude that  $G_1(\lambda_1(\widetilde{H}))>0.$  By Lemma \ref{lemma:2}, $\lambda_1(\widetilde{H})$ are the largest root of $f(x)=0$. Thus, $\lambda_1(\widetilde{H})<\lambda_1(M_{\Pi_4}(G^*))=\lambda_1(G^*)$, which contradicts the choice of $G^*$. 

\noindent\textbf{Case 4.} $h = 1$ and $a_1 \ge 1$.

Notice that $k'=g+h \ge 2$ and $a_1=\alpha-(k'-1) \ge 1$. Then, $\alpha \ge k'$. The partition $\Pi_5: V(G^*) = (T'\setminus S') \cup (T'\cap S') \cup (C_1\cap S') \cup (C_1\setminus S') \cup C_2 $ 
is an equitable partitions of $G^*$, and the corresponding equitable quotient matrix is
\[
M_{\Pi_5}(G^*)=
\begin{pmatrix}
0&k'-1 & n-\alpha-3 &\alpha-k'+1& 2 \\
1&2(k'-2) &n-\alpha-3 &2(\alpha-k'+1) & 2 \\
1&k'-1 & n-\alpha-4 & \alpha-k'+1 & 4 \\
1&2(k'-1) & n-\alpha-3 & 2(\alpha-k') & 4 \\
1& k'-1& 2(n-\alpha-3)&2(\alpha-k'+1)&1
\end{pmatrix}.
\]
Set $\det\bigl(x I_{5}-M_{\Pi_5}(G^*)\bigr)=h_2(x)$ and $G_2(x)=(x+1)f(x)-h_2(x).$ By direct calculation, we obtain
\begin{align*}
G_2(x)=(x+1)f(x)-h_2(x)=B_1x^{3}+B_2x^{2}+B_3x+B_4,
\end{align*}
where $B_1=3\alpha-3$, $B_2=28\alpha+2k'+n-8\alpha k'-\alpha n+\alpha^{2}+2k'^{2}-25$, $B_3=59\alpha+4n-24\alpha k'-8\alpha n+6k'n+4\alpha k'^{2}-2k^{2}n+8\alpha^{2}+4k'^{2}-51$ and $B_4=30\alpha-2k'+n-16\alpha k'-5\alpha n+6k'n+4\alpha k'^{2}-2k'^{2}n+5\alpha^{2}+2k'^{2}-23.$ Thus, $G_2'(x)=3B_1x^2+2B_2x+B_3$, $G_2''(x)=6B_1x+2B_2$ and $G_2'''(x)=6(3\alpha-3)>0$. Furthermore, combining with $k' \ge 2$, $\alpha-k' \ge 0$ and $n \ge k'+\alpha+3$, direct computation in  MATLAB R2018b \cite{Inc} yields $G_2'(n-1)>0$, $G_2''(n-1)>0$ and $G_2( \frac{\alpha^2+4n-4k'-3\alpha-6}{n}+n-1)>0$. By a similar argument as in Case 3 of Claim \ref{claim:4}, we can deduce that $\lambda_1(\widetilde{H})<\lambda_1(G^*)$, a contradiction. Thus, the  claim holds.

So, $G^*\cong 
K_{k'} \vee \Big(\big(K_{n-k'-\alpha-1} \vee (\alpha-1)K_1\big) \cup K_2\Big)$. This completes the proof. 

\end{proof}

\section{Proof of Theorem \ref{thm:1}~\ref{thm:main:ii}: the case $n = k'+\alpha+2 $. } \label{sec:4}
\begin{lemma}\label{lemma:4}
Let $a,b,k'$ be positive integers with $a+b=\alpha$. Define
\[G(a,b)\cong K_{k'}\lor \big((K_1\lor aK_1)\cup(K_1\lor bK_1)\big).\]
Then $\lambda_1\big(G(a,b)\big)$ attains its minimum if and only if $|a-b|\le 1$.
\end{lemma}
\begin{proof}[\bfseries Proof ]
For any positive integers $a,b,k'$ with $a+b=\alpha$, let $G(a,b)$ be the graph with minimum distance spectral radius. Without loss of generality, assume $a\ge b \ge 1$. When $a\ge b+2 \ge 3$, it remains to prove that $\lambda(G(a-1,b+1)) < \lambda(G(a,b))$. 
Let $Y$ be the perron vector of $G(a-1,b+1)$ and $V(G(a-1,b+1)) =V( K_{k'})\cup \{u\} \cup V((a-1)K_1)\cup \{w\} \cup V((b+1)K_1)$.
Here, $u$ is adjacent to each vertex in $V((a-1)K_1)$, and $w$ is adjacent to each vertex in $V((b+1)K_1)$. Let $y_v$ denoting the entry of $Y$ corresponding to vertex $v\in V(G(a-1,b+1))$.
By symmetry, we set $y_v=y_1$ for any $v\in V((a-1)K_1)$, $y_v=y_2$ for any $v\in V((b+1)K_1)$, $y_v=y_0$ for any $v\in K_{k'}$, $y_u=y_1'$ and $y_w=y_2'$.
Set $\lambda_1(G(a-1,b+1))=\lambda_1$. From $D(G(a-1,b+1))Y=\lambda_1 Y$, we obtain 
\begin{align*}
\lambda_1 y_1' &= k'y_0 + (a-1) y_1 + 2 y_2' + 2(b+1) y_2, \\
\lambda_1 y_2' &= k'y_0 + 2(a-1) y_1 + 2y_1' + (b+1) y_2,\\
\lambda_1 y_1  &= k'y_0 + 2(a-2)y_1 + 2 y_2' + 2(b+1) y_2 + y_1', \tag{10}\label{eq:10}\\
\lambda_1 y_2  &= k'y_0 + 2(a-1)y_1 + 2 y_1' + 2by_2 + y_2', \tag{11}\label{eq:11}
\end{align*}
which yield $(\lambda_1-a+3)y_1=(\lambda_1+1)y_1'$, $(\lambda_1+2)(y_2-y_1')=(a-1)y_1-y_2'$ and $(\lambda_1+2)(y_2'-y_1')=(a-1)y_1-(b+1)y_2$.
From \eqref{eq:10} and \eqref{eq:11}, we obtain $\lambda_1 > 2(a-2)$ and $\lambda_1 > 2b$. Consequently, $\lambda_1-a+3 > a-1> 0$ and $\lambda_1 > b+1 >0$. Therefore, we conclude that 
\[\left[\lambda_1+2-\frac{b+1}{\lambda_1+2}\right](y_2'-y_1')=\frac{(a-b-2)(\lambda_1+1)}{\lambda_1-a+3} y_1'\ge 0.\]
Since $\lambda_1+2-\frac{b+1}{\lambda_1+2}=\frac{\lambda_1^2+4\lambda_1-b+3}{\lambda_1+2}>0$, this demonstrates that $y_2'\ge y_1'$. Thus,
\begin{align*}
\frac{1}{2}(\lambda_1(G(a,b)) -\lambda_1(G(a-1,b+1)))& \ge\frac{1}{2}Y^T (D(G(a,b)) - D(G(a-1,b+1))) Y\\ &=y_2(y_2'-y_1')\ge 0.
\end{align*}
If $\lambda_1(G(a,b))=\lambda_1(G(a-1,b+1))$, then $Y$ is also a Perron vector of $G(a,b)$, which implies $D(G(a,b))Y=\lambda_1(G(a,b)) Y$ and yields a contradiction. Thus we obtain $\lambda_1(G(a,b))>\lambda_1(G(a-1,b+1))$, which contradicts the minimality  of $\lambda_1(G(a,b))$.  Therefore, $|a-b|\le 1$.
\end{proof}
\begin{proof}[\bfseries Proof of Theorem \ref{thm:1}~\ref{thm:main:ii}]Suppose that $G\in\Gcal^{\alpha}_{n,k'}$ is a graph that attains the minimum distance spectral
radius, where $n=k'+\alpha+2$. Let $S''$ be an independent set of $G$ with $|S''|=\alpha$, and let $T''$ be an essential vertex cut of $G$ with $|T''|=k'$ such that $G-T''$ contains $r$ connected components $B_1, B_2, \dots, B_r$, at least two of which are non-trivial connected components. For convenience, we set $|T''\cap S''|=g'$, $|T''\setminus S''|=h'$ , $|V_{B_i}\cap S''|=a_i'$ and $|V_{B_i}\setminus S''|=b_i'$ for $1\leq i\leq r$. By arguments similar to the proofs of Claim \ref{claim:1}  and Claim \ref{Claim:2} in Theorem \ref{thm:1}~\ref{thm:main:i}, we obtain $r = 2$ , $G[T'']\cong K_h'\vee g'K_1$, $B_i\cong K_{b_i'}\vee a_i'K_1$ and $G[V_{B_i}\cup T'']\cong K_{h'+b_i'}\vee(g'+a_i')K_1$ for $1\leq i\leq 2$. Set $K_{h'} = \{u_1, u_2, \dots, u_{h'}\}$, $T'' \cap S''= \{u_1', u_2', \dots, u'_{g'}\}$, $K_{b_i'} = \{w_i^1, w_i^2, \dots, w_i^{b_i'}\}$ and $a_i'K_i = \{v_i^1, v_i^2, \dots,v_i^{a_i'}\}$ for $1\leq i\leq 2$.
 It is clear that $b_1'+b_2'=g'+2$. We assert that $g'=0$. 
 Otherwise, $g' \ge 1$. If $b_1'= b_2'=2$ and $a_1'=a_2'=0$, set  $H'=G+\{u_1'u_2'\}$, then $\alpha(H')=2=\alpha(G)$ and $\lambda_1(G)> \lambda_1(H')$ by Lemma \ref{lemma:1}, which yields a contradiction. 
 Assume that $E_1' = \{w_1^i v \mid v \in V_{B_2}, 2 \le i \le b_1' \} + \{w_2^i v \mid v \in V_{B_1}, 2 \le i \le b_2' \}$ and
 $E_2' = \{u_i' w_2^1 \mid 1 \le i \le b_1'-1\} + \{u_i' w_1^1 \mid b_1' \le i \le g'\}$ . 
 Let $H = G - E_2' + E_1'$. Obviously, $H\cong K_{k'}\lor \big((K_1\lor (a_1'+b_1'-1)K_1)\cup(K_1\lor (a_2'+b_2'-1)K_1)\big)$ , $H\in\Gcal^{\alpha}_{n,k'}$ and $\lambda_1(H)>\lambda_1(K_{k'+1})=k'>0$ follows from Lemma \ref{lemma:1}.
 Let $X$ be the Perron vector of $D(H)$, with \(x_v\) denoting the entry of $X$ corresponding to vertex $v\in V(H)$. By symmetry, set $x_v=x_0$ for any $ v \in V(K_{k'})$, $x_v=x_1'$ for any $v \in V((a_1'+b_1'-1)K_1),x_v=x_2'$ for any $v \in V((a_2'+b_2'-1)K_1)$. Therefore, by $D(H)X =\lambda_1(H)X$, we get
 \begin{align*}
\lambda_1(H) x_0 &= (k'-1)x_0 + x_{w_1^1} + (a_1'+b_1'-1)x_1' + x_{w_2^1}+ (a_2'+b_2'-1)x_2', \tag{12}\label{eq:12}\\
\lambda_1(H) x_{w_1^1} &= k'x_0 + (a_1'+b_1'-1)x_1'+ 2x_{w_2^1}+ 2(a_2'+b_2'-1)x_2', \\
\lambda_1(H)x_{w_2^1}  &= k'x_0 +2 x_{w_1^1} + 2(a_1'+b_1'-1)x_1' + (a_2'+b_2'-1)x_2', \\
\lambda_1(H)x_1'  &= k'x_0 + x_{w_1^1} + 2(a_1'+b_1'-2)x_1' +2x_{w_2^1}+ 2(a_2'+b_2'-1)x_2', \tag{13}\label{eq:13}\\
\lambda_1(H)x_2'  &= k'x_0 +2 x_{w_1^1} + 2(a_1'+b_1'-1)x_1' +x_{w_2^1}+ 2(a_2'+b_2'-2)x_2',\tag{14}\label{eq:14}
\end{align*}
from which we obtain that
\begin{align*}
x_{w_1^1} - x_0 &= \dfrac{x_{w_2^1}+(a_2'+b_2'-1)x_2'}{\lambda_1(H)+1} > 0, \tag{15}\label{eq:15} \\
x_1' - x_{w_1^1} &= \dfrac{(a_1'+b_1'-2)x_1'}{\lambda_1(H)+1} \ge 0, \tag{16}\label{eq:16} \\
2x_0 - x_1' &= \dfrac{kx_0 + x_{w_1^1} + x_1'}{\lambda_1(H)+1} > 0 \tag{17}\label{eq:17} \\
\end{align*}
and
\[(\lambda_1(H)+1)\big(x_{w_1^1}+x_{w_2^1}-2x_0\big) = \big(\lambda_1(H)+1-k\big)x_0. \tag{18}\label{eq:18}\]
Using computations analogous to \eqref{eq:15}--\eqref{eq:17}, we readily obtain $2x_0>x_2'\ge x_{w_2^1}>x_0$. Moreover,
\begin{align*}
\frac{1}{2}(\lambda_1(G) -\lambda_1(H))& \ge \frac{1}{2} X^T (D(G) - D(H)) X \\ 
&=(b_1'-1)(b_2'-1)x_0^2+(b_1'-1)(a_2'x_0x_2'+x_0x_{w_2^1}-x_1'x_{w_2^1})\\&+(b_2'-1)(a_1'x_0x_1'+x_0x_{w_1^1}-x_2'x_{w_1^1}).
\end{align*}
For convenience, set $(b_1'-1)(b_2'-1)x_0^2+(b_1'-1)(a_2'x_0x_2'+x_0x_{w_2^1}-x_1'x_{w_2^1})+(b_2'-1)(a_1'x_0x_1'+x_0x_{w_1^1}-x_2'x_{w_1^1})=\Phi$.
If $b_1'=1$ or $b_2'=1$, then without loss of generality we may assume that $b_1'=1$. Consequently, $b_2'=g'+1 \ge 2$ and $a_1' \ge 1$. Combining these inequalities with $x_2'< 2x_0$, $x_1'< 2x_0$,
 $x_{w_2^1}\le x_2'$ and $x_{w_1^1}\le x_1'$, one can readily deduce that $\Phi > 0$. It follows that $\lambda_1(G) > \lambda_1(H)$, which contradicts the minimality of $\lambda_1(G)$. Hence, $b_1'> 1$ and $b_2'> 1$.
The sign of $ \Phi$ is considered in two cases. 

\medskip
\noindent\textbf{Case 1.} $a_1' \neq 0$ or $a_2' \neq 0$.

We now consider several subclasses of this case.

\noindent\textbf{Subcase 1.1.} $a_1 '\ge 1$  and $a_2' \ge 1$.

By analysis similar to the above, $\Phi>0 $ is easily obtained. Hence, $\lambda_1(H)<\lambda_1(G)$, a contradiction.
\medskip

\noindent\textbf{Subcase 1.2.} Either $a_1'=0$ or $a_2'=0$, but not both.

Without loss of generality, we assume that $a_2'= 0$. Thus, $a_1'\ge 1$. When $b_2'\ge 3$ or $x_0^2-x_{w_2^1}(x_1'-x_0) \ge 0$, combining $2x_0>x_i'\ge x_{w_i^1}>x_0$ for $i\in\{1,2\}$ with $b_1'>1$, we deduce that $ \Phi > 0$. Next consider the case $b_2'= 2$ and $x_0^2-x_{w_2^1}(x_1'-x_0)<0$.
By symmetry, we get $x_{w_2^1}=x_2'$. Combining Equations \eqref{eq:12}, \eqref{eq:13} and \eqref{eq:14}, we have
$(\lambda_1(H)-a_1'-b_1'+3)x_1'=(\lambda_1(H)+1)x_0+2x_2'$ and $(\lambda_1(H)+3)x_2'=(2\lambda_1(H)+2-k')x_0$, from which we conclude
\begin{align*}
&(\lambda_1(H) +3)(x_0^2+x_0x_2'-x_2'x_1')\\
&=(\lambda_1(H) +3)x_0^2+(\lambda_1(H) +3)x_0x_2'-(\lambda_1(H) +3)x_2'x_1'\\
&= (\lambda_1(H)-a_1'-b_1'+3)x_0x_1'+2x_0^2-2x_2'x_1'+(\lambda_1(H)+1)x_2'(x_0-x_1')\\
&=(\lambda_1(H)-a_1'-b_1'+3)x_0x_1' -(2\lambda_1(H)+2-k)x_0x_1' + 2x_0^2+(\lambda_1(H)+1)x_2'x_0\\
&=(k-a_1'-b_1'+2)x_0x_1'-(\lambda_1(H)+1)x_0x_1'+2x_0^2+(\lambda_1(H)+1)x_2'x_0\tag{19}\label{eq:19}. \\
\end{align*}
Moreover, by \eqref{eq:13} and \eqref{eq:14}, we obtain
\begin{align}
x_1'-x_2'=\frac{x_1'-x_{w_1^1}}{\lambda_1(H)+3} \tag{20}\label{eq:20}.
\end{align}
Combining \eqref{eq:19} and \eqref{eq:20}, we obtain
\begin{align*}
\Phi=&(b_1'-1)(x_0^2-x_2'(x_1'-x_0 ))+(a_1'x_0x_1'-x_{w_1^1}(x_2'-x_0)\\
\ge&(k-1)(x_0^2-x_2'(x_1'-x_0 ))+(a_1'+b_1'-k')x_0x_1'-x_{w_1^1}(x_2'-x_0)(\text{since }  1<b_1' \le k' ).\\
=&-(\lambda_1(H) +4-k')(x_0^2-x_2'(x_1'-x_0 ))+2x_0^2+2x_0x_1'\\
&-(\lambda_1(H)+1)x_0(x_1'-x_2')-x_{w_1^1}(x_2'-x_0)\\
>&2x_0^2+2x_0x_1'-\frac{(\lambda_1(H)+1)(x_1'-x_{w_1^1})x_0}{\lambda_1(H)+3} -2x_0^2\\
&\quad (\text{since }\lambda_1(H) > k' ,2x_0>x_2'>x_0 ,x_1'\ge x_{w_1^1}, x_{w_1^1}<2x_0  ).\\
>&2x_0^2+2x_0x_1'-x_0x_1'-2x_0^2\\
=&x_0x_1'>0.
\end{align*}
Then $\Phi>0$ and $\lambda_1(H)<\lambda_1(G)$, a contradiction.

\medskip
\noindent\textbf{Case 2.} $a_1'=a_2'=0$. 

Since $b_i'\ge 2$ for $1\le i \le 2$, we readily obtain \[\min_{1\leq i\leq n}S_i(D(H))=k'+b_1'+b_2'-1< k'+2(b_1'+b_2')-3=\max_{1\leq i\leq n}S_i(D(H)).\] Thus, by Lemma \ref{lemma:3}, we have $k'+b_1'+b_2'-1<\lambda_1(H)<k'+2(b_1'+b_2')-3$. Notice that $k'\ge b_1'+b_2'-2$. Then, $k'+2(b_1'+b_2')-3>\lambda_1(H)>k'+b_1'+b_2'-1\ge 2(b_1'+b_2')-3.$

\medskip
\noindent\textbf{Subcase 2.1.} $b_1'\ge 3$ and $b_2'\ge 3$.

 It follows from \eqref{eq:18} that
 \begin{align*}
 (\lambda_1(H) +1)(x_{w_1^1}-x_0)=&(\lambda_1(H)+1 -k')x_0-(\lambda_1(H) +1)(x_{w_2^1}-x_0)\\
 \le &(\lambda_1(H) +3-b_1'-b_2')x_0-(\lambda_1(H) +1)(x_{w_2^1}-x_0)(\text{since } k' \ge b_1'+b_2'-2).
\end{align*}
Integrating with \eqref{eq:16} , then
\[x_1'-x_0=\frac{(\lambda_1(H) +1)x_{w_1^1}}{\lambda_1(H) +3-b_1'}-x_0=\frac{(\lambda_1(H) +1)(x_{w_1^1}-x_0)+(b_1'-2)x_0}{\lambda_1(H) +3-b_1'}.\]
Thus,
\begin{align*}
x_{w_2^1}(x_1'-x_0) \le & \frac{(x_0+x_{w_2^1}-x_0)((\lambda_1(H) +1-b_2')x_0-(\lambda_1(H) +1)(x_{w_2^1}-x_0))}{\lambda_1(H) +3-b_1'}\\
=&\frac{(\lambda_1(H) +1-b_2')x_0^2-b_2'x_0(x_{w_2^1}-x_0)-(\lambda_1(H) +1)(x_{w_2^1}-x_0)^2}{\lambda_1(H) +3-b_1'}\\
<& \frac{(\lambda_1(H) +1-b_2')x_0^2}{\lambda_1(H) +3-b_1'}.
\end{align*}
 Similarly,
\[x_{w_1^1}(x_2'-x_0) <  \frac{(\lambda_1(H) +1-b_1')x_0^2}{\lambda_1(H) +3-b_2'}.\]
For convenience, let $p=b_1'-1 \ge2$, $q=b_2'-1 \ge 2$ and $s=p+q \ge 4$. Hence,
\begin{align*}
\Phi=&(b_1'-1)(b_2'-1)x_0^2-(b_1'-1)x_{w_2^1}(x_1'-x_0)-(b_2'-1)x_{w_1^1}(x_2'-x_0)\\
>&x_0^2(pq-\frac{p(\lambda_1(H)-q)}{\lambda_1(H) +2-p}- \frac{q(\lambda_1(H) -p)}{\lambda_1(H) +2-q})\\
=&\frac{F(\lambda_1(H)+1)x_0^2}{(\lambda_1(H) +2-p)(\lambda_1(H) +2-q)},
\end{align*}
where 
$F(\lambda_1(H)+1)=(pq-p-q)(\lambda_1(H) +1)^2-pq(p+q-6)((\lambda_1(H) +1)+p^2q^2-2p^2q-2pq^2+pq+p+q$. Then by simple calculation,
\[F'(\lambda_1(H)+1)=2(pq-s)(\lambda_1(H)+1)-pq(s-6).\]
Obviously, $F'(\lambda_1(H)+1)>0$ whenever $s\le 6$. For $s> 6$, we have 
\begin{align*}
F'(\lambda_1(H)+1)>&2(pq-s)(2s+2)-pq(s-6)\quad (\text{since }\lambda_1(H)+1 > 2s+2)\\
=&pq(3s+10)-4s(s+1)\\
\ge&2((s+1)^2-21)\quad (\text{since } pq\ge 2s-4)\\
>&0.
\end{align*}
As $p\ge2$ and $q\ge2$, it follows that $F(2s+2)=2(p-2)^3(q-2)+5(p-2)^2(q-2)^2+36(p-2)(q-2)(p+q-4)+20(p-2)^2+2(p-2)(q-2)^3+161(p-2)(q-2)+111(p+q-4)+20(q-2)^2+72>0$. Furthermore, $F(\lambda_1(H)+1)>F(2s+2)>0$. Thus, $\Phi>0$ and $\lambda_1(H)<\lambda_1(G)$, a contradiction.

\medskip
\noindent\textbf{Subcase 2.2.} $\min\{b_1',b_2'\}=2$ and $b_1'+b_2' \ge 5$.

 Without loss of generality, assume that $2 = b_2'  < b_1'$ . By symmetry, we get $x_{w_2^1}=x_2'$. From \eqref{eq:12}, \eqref{eq:14},\eqref{eq:15}, $x_{w_1^1}\le x_1'$ and $x_2'<2x_0$, we obtain
\begin{align*}
x_{w_1^1}(x_2'-x_0)
=\frac{x_{w_1^1}(x_{w_1^1}+(b_1-1)x_1')}{\lambda_1(H) +1}\le\frac{b_1'x_1'x_{w_1^1}}{\lambda_1(H) +1} <\frac{(\lambda_1(H) +5)b_1'x_0x_1'}{(\lambda_1(H) +1)^2}
\end{align*}
and
\begin{align*}
\frac{x_{w_1^1}}{x_1'}=\frac{\lambda_1(H) +3-b_1'}{\lambda_1(H) +1}.
\end{align*}
Then,
\begin{align*}
\frac{x_{w_1^1}+x_1'+2x_0}{x_1'}>\frac{\lambda_1(H) +3-b_1'}{\lambda_1(H) +1}+2=3-\frac{b_1'-2}{\lambda_1(H) +1}\quad(\text{since }2x_0 > x_1').
\end{align*}
Combining this with \eqref{eq:19}, \eqref{eq:20}, $x_1'\ge x_{w_1^1}$ and $k' \ge b_1'$, we have
\begin{align*}
(\lambda_1(H) +3)(x_0^2+x_0x_2'-x_2'x_1')
>&2x_0x_1'-x_0(x_1'-x_{w_1^1})+2x_0^2\\
=&x_0(x_1'+x_{w_1^1}+2x_0)\\
>&\frac{3x_1'x_0(\lambda_1(H) +1)-(b_1'-2)x_0x_1'}{\lambda_1(H) +1}.
\end{align*}
To simplify notation, let $m = b_1'-1\ge 2$. Hence,
\begin{align*}
\Phi=&m(x_0^2+x_0x_2'-x_2'x_1')-x_{w_1^1}(x_2'-x_0)\\
>&\frac{mx_0x_1'(3\lambda_1(H)+4-m)}{(\lambda_1(H)+1)(\lambda_1(H)+3)}-\frac{(\lambda_1(H)+5)(m+1)x_0x_1'}{(\lambda_1(H)+1)^2}\\
=&\frac{F_1(\lambda_1(H)+1)x_0x_1'}{(\lambda_1(H)+1)^2(\lambda_1(H)+3)},
\end{align*}
where
\begin{align*}
 F_1(\lambda_1(H)+1)=&(2m-1)(\lambda_1(H)+1)^2-(m^2+5m+6)(\lambda_1(H)+1)-8m-8\\
 &=(2m-1)(\lambda_1(H)-2m-3)^2+7(m-1)(m+2)(\lambda_1(H)-2m-3)\\
 &+6m^3+14m^2-24m-48\\
 &>0\quad(\text{since }m \ge 2 ,\lambda_1(H)> 2b_1'+1=2m+3 ).
\end{align*}
Therefore $\Phi>0$ and $\lambda_1(H)<\lambda_1(G)$, a contradiction. Hence $g'=0$, i.e., $S''\cap T''=\emptyset$ and $b_1'+b_2'=g'+2=2$. Combining this with $b_i'\ge 1$, we get $b_1'= b_2'=1$.
By the Lemma \ref{lemma:4} and the minimality of $\lambda_1(G)$, we conclude that $G\cong 
K_{k'} \vee ((K_1 \vee \lfloor \alpha/2 \rfloor K_1) \cup (K_1 \vee \lceil \alpha/2 \rceil K_1))
$.

\end{proof}

\end{document}